\documentclass[reqno,11pt]{amsart}

\usepackage{amscd}
\usepackage{amsmath}
\usepackage{amssymb}
\usepackage[american]{babel}
\usepackage{bbm}
\usepackage{bookmark}
\usepackage{cmap} 
\usepackage{dsfont}
\usepackage{enumerate}
\usepackage{epigraph}
\usepackage[mathscr]{euscript}
\usepackage[myheadings]{fullpage}
\usepackage{graphicx}
\usepackage[geometry]{ifsym}
\usepackage{mathabx}
\usepackage{accents}
\usepackage{mathrsfs}
\usepackage{mathtools}
\usepackage{refcount}
\usepackage{setspace}
\usepackage{stmaryrd}
\usepackage{thinsp}
\usepackage{verbatim}
\usepackage[all,2cell]{xy} \UseTwocells
\usepackage{xr}
\usepackage[multiple]{footmisc}
\usepackage{hyperref}

\numberwithin{equation}{subsection}

\newtheorem{thm}{Theorem}[subsection]
\newtheorem*{thm*}{Theorem}

\newtheorem*{cor*}{Corollary}
\newtheorem{lem}[thm]{Lemma}

\newtheorem{prop}[thm]{Proposition}
\newtheorem{prop-const}[thm]{Proposition-Construction}

\newtheorem*{conjecture*}{Conjecture}

\newtheorem*{princ*}{Principle}

\theoremstyle{remark}
\newtheorem{rem}[thm]{Remark}
\newtheorem{example}[thm]{Example}

\newtheorem{defin}[thm]{Definition}

\newcounter{steps}[thm]
\newtheorem{step}[steps]{Step}

\newcommand{\xar}[1]{\xrightarrow{#1}}
\newcommand{\rar}[1]{\xar{#1}}
\newcommand{\isom}{\rar{\simeq}}

\newcommand{\into}{\hookrightarrow}

\newcommand{\bA}{{\mathbb A}}

\newcommand{\bD}{{\mathbb D}}

\newcommand{\bZ}{{\mathbb Z}}

\newcommand{\sC}{{\EuScript C}}
\newcommand{\sD}{{\EuScript D}}

\newcommand{\sF}{{\EuScript F}}
\newcommand{\sG}{{\EuScript G}}

\newcommand{\sO}{{\EuScript O}}

\newcommand{\on}{\operatorname}
\newcommand{\ol}[1]{\overline{#1}{}}

\newcommand{\Ker}{\on{Ker}}

\newcommand{\Hom}{\on{Hom}}

\newcommand{\Spec}{\on{Spec}}

\newcommand{\id}{\on{id}}

\newcommand{\gr}{\on{gr}}

\renewcommand{\dot}{\bullet}

\newcommand{\Image}{\on{Image}}

\newcommand{\Gal}{\on{Gal}}

\newcommand{\colim}{\on{colim}}

\renewcommand{\lim}{\on{lim}}
\newcommand{\Ind}{\mathsf{Ind}}
\newcommand{\Pro}{\mathsf{Pro}}

\newcommand{\heart}{\heartsuit}

\newcommand{\QCoh}{\mathsf{QCoh}}

\renewcommand{\SS}{\on{SS}}

\renewcommand{\subset}{\subseteq}

\makeatletter
\newcommand{\biggg}{\bBigg@{4}}
\newcommand{\Biggg}{\bBigg@{5}}
\makeatother

\date{\today}

\begin{document}

\frenchspacing

\setlength{\epigraphwidth}{0.9\textwidth}
\renewcommand{\epigraphsize}{\footnotesize}

\title{A generalization of the $b$-function lemma}

\author{Sam Raskin}

\address{Massachusetts Institute of Technology, 
77 Massachusetts Avenue, Cambridge, MA 02139.} 
\email{sraskin@mit.edu}

\begin{abstract}

We establish some cohomological bounds
in $D$-module theory that are known in the holonomic
case and folklore in general. The method rests on a generalization
of the $b$-function lemma for non-holonomic $D$-modules. 

\end{abstract}

\maketitle

\setcounter{tocdepth}{1}
\tableofcontents

\section{Introduction}

\subsection{}

This note studies how $D$-module operations interact with
singular support. The main technical result, Theorem \ref{t:hol},
shows that $D$-module operations preserve the natural
numerical obstruction to holonomicity. 
This result generalizes the usual preservation of 
holonomic $D$-modules under such operations, which is
essentially equivalent to the $b$-function lemma: 
see \cite{kashiwara} or \cite{bernstein}.

\subsection{Affine morphisms}

As an application, we show in Theorem \ref{t:aff}
that $f_!$ is left $t$-exact for an affine morphism $f:X \to Y$.

This is certainly an old folklore result. 
Of course it is standard for holonomic $D$-modules,
where it is a consequence of the usual $b$-function lemma.
It is also easy to show for $Y = \Spec(k)$,
or for a map of curves (e.g., an open embedding).
Otherwise, it does not seem to follow from existing 
foundational results in the literature, which is quite
surprising for something so basic.

We remark that the formulation of this result
does not quite make sense, since $f_!(\sF)$ does not typically make 
sense as a $D$-module (although it always does if $\sF$ is holonomic).
One can rectify this in one of two ways: one can ask to
show that if $\sF$ is in cohomological degrees $\geq 0$
and $f_!(\sF)$ is defined, $f_!(\sF)$ is in degrees $\geq 0$; or
one can work with pro-complexes. We use the latter technique, since
it is somewhat more general. For technical reasons,
we only work with coherent $D$-modules $\sF$.

Applying this result for non-holonomic $D$-modules is actually useful in geometric representation theory.
The point is that in many settings typical
of the subject,
$f_!$ is defined on some non-holonomic 
$D$-modules of interest even when $f$ is affine. 
For example,
this occurs for the Fourier-Deligne transform, and the results
here can be used to show its $t$-exactness in a conceptual way.\footnote{C.f.
\cite{kernels} \S 1.8. Note that \emph{loc. cit}. implicitly assumes the
left $t$-exactness of $f_!$ for affine $f$.}
For an application of such results to Lie theory, see \cite{w-alg}. 

\subsection{Notation}

We let $k$ denote a field of characteristic zero.
We use the phrase ``category" to mean $\infty$-category
wherever appropriate. (This language is used only very
mildly.) 

By a \emph{variety}, we mean a reduced,
separated, finite type $k$-scheme.

For $X$ a variety over $k$, we let $D(X)$ denote the DG category
of $D$-modules on $X$. 
We let $D(X)^{\geq i}$ and $D(X)^{\leq i}$ respectively
denote the subcategories of complexes $\sF \in D(X)$
with $H^j(\sF) = 0$ for $j<i$ and $j>i$ respectively. 
We let $D(X)^{\heart} = D(X)^{\geq 0} \cap D(X)^{\leq 0}$ 
denote the heart of the $t$-structure, i.e., the abelian
category of $D$-modules. We let $\tau^{\geq i}$
and $\tau^{\leq i}$ denote the corresponding truncation
functors.

For $f:X \to Y$, we let $f^!:D(Y) \to D(X)$ and $f_{*,dR}:D(X) \to D(Y)$
denote the $D$-module pullback and pushforward operations.
We let $f_!$ and $f^{*,dR}$ denote their left adjoints where
appropriate.

We let $D(X)^c \subset D(X)$ denote the DG subcategory
of coherent complexes, i.e., bounded complexes with
coherent (i.e., locally finitely generated) cohomology groups.
Recall that $D(X)$ is compactly generated, i.e.,
$D(X) = \Ind(D(X)^c)$. We let $\bD:D(X)^c \isom D(X)^{c,op}$ denote
the Verdier duality functor.

\subsection{Acknowledgements} 

We are grateful to Dennis Gaitsgory for some useful
correspondence on the subject of this note.
The methods owe a great deal to \cite{bernstein},
\cite{ginzburg}, and \cite{kashiwara}.

\section{Holonomic defect}

\subsection{} In this section, we introduce a generalization of the
holonomic condition on a $D$-module and show that it is
preserved under $D$-module operations. 

The method is standard. 
The main point is Lemma \ref{l:b-fn}, which is a generalization of the
fact that pushforward along an open embedding
preserves holonomic objects, which is essentially equivalent to the
usual $b$-function lemma. The main difference is that we 
cannot use finite length methods.

The presentation is based on \cite{kashiwara} and \cite{ginzburg}.

\subsection{Gabber-Kashiwara-Sato (GKS) filtration}

We begin by reviewing some material from \cite{ginzburg} \S 1.

Let $X$ be a variety and let $\sF \in D(X)^{\heart}$
be a given $D$-module.

\begin{defin}

For an integer $i$, we let:

\[
F_i^{GKS} \sF \coloneqq 
\Image(H^0(\bD \tau^{\geq -i} \bD \sF) \to \sF). 
\]

\end{defin}

\begin{rem}

By definition, $\bD \tau^{\geq -i} \bD \sF \in D(X)$ means the obvious
thing if $\sF$ is coherent, and in general, we understand this
expression to commute with filtered colimits. 
(It is equivalent interpret this more literally
and consider 
$\bD$ as an equivalence between $D(X)$ and the
DG category of pro-coherent $D$-modules,
equipped with the $t$-structure of \S \ref{ss:pro-tstr}.)

\end{rem}

Note that $F_{\dot}^{GKS}$ is an increasing filtration on
$\sF$. Because $\bD \sF$ is in cohomological degrees $[-\dim X,0]$,
we have $F_i^{GKS} \sF = 0$ for $i<0$, and 
$F_i^{GKS} \sF = \sF$ for $i \geq \dim X$.
Formation of the GKS filtration is functorial for $D$-module
morphisms, i.e., a map
$\sF_1 \to \sF_2 \in D(X)^{\heart}$
sends $F_i^{GKS} \sF_1$ to $F_i^{GKS} \sF_2$.

Note that if $\sF = \colim_j \sF_j$ is a filtered colimit in $D(X)^{\heart}$,
then $F_i^{GKS} \sF = \colim_j F_i^{GKS} \sF_j$.

\begin{lem}

Formation of $F_{\dot}^{GKS}$ commutes with
open restriction and pushforwards along closed embeddings.

\end{lem}

\begin{proof}

Each of these functors is $t$-exact and commutes with Verdier duality.

\end{proof}

Therefore, many results about this 
filtration reduce to the case of smooth
$X$ by taking Zariski local closed embeddings into
affine space. The key property in the smooth case is:

\begin{thm}\label{t:gks}

If $X$ is smooth, then a
local section $s$ of $\sF$ lies in $F_i^{GKS} \sF$ if and
only if the $D$-module generated by it has 
singular support with dimension $\leq \dim X + i$.\footnote{
We regard $\dim X$ as a locally constant function on $X$
if $X$ is not equidimensional.}

\end{thm}

See \cite{ginzburg} Proposition V.14. 
Note that it is equivalent to say that $F_i^{GKS} \sF$ is the maximal
submodule of $\sF$ with singular support of dimension 
$\leq \dim X+i$.

\subsection{Holonomic defect}

For $\delta \in \bZ^{\geq 0}$, we 
say $\sF \in D(X)^{\heart}$ has \emph{holonomic defect $\delta$}
if $F_{\delta}^{GKS} \sF = \sF$. 

\begin{rem}

If $\sF$ has holonomic defect $\delta$, then it also has
holonomic defect $1+\delta$.

\end{rem}

\begin{example}

A coherent $D$-module 
$\sF$ has holonomic defect $0$ if and only if $\sF$ is holonomic.
Indeed, this follows by reduction to the smooth case
and Theorem \ref{t:gks}.

\end{example}

\begin{example}

Every $\sF$ has holonomic defect $\dim X$.

\end{example}

\begin{example}

If $X$ is smooth and $\sF$ is coherent, then
by Theorem \ref{t:gks}, $\sF$ has holonomic defect $\delta$
if and only if $\sF$ has singular support with dimension
$\leq \dim X + \delta$.

\end{example}

\begin{lem}\label{l:defect-sub/quot/ext}

The subcategory of $D(X)^{\heart}$ consisting of
objects with holonomic defect $\delta$ is closed under
submodules, quotient modules, and extensions.

\end{lem}

\begin{proof}

The argument reduces to the case of $X$ smooth, and then
follows from Theorem \ref{t:gks} and standard facts about
singular support.

\end{proof}

\begin{lem}\label{l:defect-colim}

Holonomic defect is preserved under filtered colimits,
and $\sF \in D(X)^{\heart}$ has holonomic defect 
$\delta$ if and only if $\sF = \colim \sF_i$ with
$\sF_i$ coherent of holonomic defect $\delta$.

\end{lem}

\begin{proof}

The first part is clear since formation of $F_{\dot}^{GKS}$ commutes
with filtered colimits. For the second part,
write $\sF = \colim_i \sF_i^{\prime}$ with 
$\sF_i^{\prime}$ coherent, and then set 
$\sF_i = F_{\delta}^{GKS} \sF_i^{\prime}$.

\end{proof}

\subsection{}\label{ss:defect-cplx}

More generally, for $\sF \in D(X)$ a complex of $D$-modules, 
we say that $\sF$ has
\emph{holonomic defect} $\delta$ if all of its cohomology groups do.
By Lemmas \ref{l:defect-sub/quot/ext} and \ref{l:defect-colim},
this defines a DG subcategory of $D(X)$ closed under colimits.

\subsection{}

The following is the main result of this section.

\begin{thm}\label{t:hol}

Holonomic defect is preserved under $D$-module operations.
That is, if $f:X \to Y$ is a morphism and $\sF \in D(X)$
(resp. $\sG \in D(Y)$) has holonomic defect $\delta$,
then $f_{*,dR}(\sF)$ (resp. $f^!(\sG)$) does as well.
Moreover, for $\sF$ coherent as above, $\bD \sF$ has holonomic
defect $\delta$ as well.

\end{thm}

This theorem generalizes the preservation of holonomic objects
under $D$-module operations, so the proof must follow similar lines. 
It is given below.

\subsection{Verdier duality}

The compatibility with Verdier duality in Theorem \ref{t:hol}
is well-known. Indeed, the result immediately reduces to 
$X$ being smooth, and then we have:

\begin{prop}\label{p:dual-ss}

For $\sF \in D(X)^{\heart}$ with singular support of dimension 
$\leq \dim X+i$, we have $H^{-j} \bD \sF = 0$ unless
$0\leq j \leq i$. Moreover, $H^{-j} \bD \sF$ has holonomic defect $j$.

\end{prop}

See e.g. \cite{kashiwara} Theorem 2.3.

\subsection{Affine open embeddings}

The main case of Theorem \ref{t:hol} is pushforward along an
open embedding.

\begin{lem}\label{l:b-fn}

Let $X$ be a connected, smooth variety and let $f:X \to \bA^1$ be 
a function.
Let $U = \{f \neq 0\}$ be the corresponding basic open and let
$j:U \into X$ denote the corresponding affine open embedding.

Then $j_{*,dR}$ preserves holonomic defect.

\end{lem}

\begin{proof}

\step

We may obviously assume $X$ is connected and affine
and that $f$ is non-constant.

We abuse notation slightly in letting
$D_X$ and $D_U$ denote the respective rings of differential
operators (as opposed to the sheaves of differential operators). 

Let $\sF$ be a $D_U$-module. Because we are working
with modules rather than sheaves, considering $\sF$ as a 
$D_X$-module by restriction is the same as considering the sheaf
$j_{*,dR}(\sF) \in D(X)^{\heart}$.

For $s \in \sF$, we write $\SS_U(s) \subset T^* U$ for the
singular support of $D_U \cdot s$ and $\SS_X(s) \subset T^* X$
for the singular support of $D_X \cdot s$. Note that
$\SS_X(s)|_{T^* U} = \SS_U(s)$. We always understand
singular support as a \emph{reduced} subscheme. 

We want to show that if every section $s \in \sF$ has
$\dim \SS_U(s) \leq \dim U + \delta = \dim X + \delta$, then
the same is true of $\dim SS_X(s)$.

\step\label{st:gks}

First, we observe that there is a $D_X$-submodule
$\sG \subset \sF$ such that
every section of $\sG$ has singular support 
with dimension $\leq \dim X + \delta$, and
which is a \emph{lattice}, i.e.,
$\sG \otimes_{\sO_X} \sO_U \isom \sF$.

Indeed, we can take $\sG = F_{\delta}^{GKS} \sF$, where
the GKS filtration is with $\sF$ considered as a $D_X$-module.
Because the GKS filtration commutes with open restriction,
we must have $\sF_0|_U = \sF$.

(Note that by Theorem \ref{t:gks}, we are trying to show that $\sG = \sF$.)

\step

Let $\lambda$ be an indeterminate. We write
$\bA_{\lambda}^1$ for $\Spec(k[\lambda])$. 
We let $k^{\prime}$ denote the fraction field $k(\lambda)$ 
of $k[\lambda]$. We use similar
notation for a base-change to $k^{\prime}$; e.g., $X^{\prime}$, or
$\sF^{\prime}$, etc. We always consider $X^{\prime}$ and
$U^{\prime}$ as schemes over $k^{\prime}$, so
e.g. their cotangent bundles are understood
relative to $k^{\prime}$, and $D_X^{\prime} = D_{X^{\prime}}$.

Recall that $U \coloneqq \{f \neq 0\}$. 
Then we have the $D_U^{\prime}$-module 
$``f^{\lambda}" \cdot \sF^{\prime}$, the tensor product
of the usual $D$-module $``f^{\lambda}"$ with $\sF^{\prime}$.

\step 

We first show that the result is true for
$``f^{\lambda}" \cdot \sF^{\prime}$, i.e., that every section
has $SS_{X^{\prime}}$ with dimension $\leq \dim X + \delta$.

First, note that the singular support in $U^{\prime}$ of any section
has dimension $\leq \dim X + \delta$: this follows because
$``f^{\lambda}"$ is lisse on $U^{\prime}$.

We have a canonical element of the Galois group
$\gamma \in \Gal(k^{\prime}/k)$ sending $\lambda \mapsto \lambda+1$.
Of course, anything obtained by extension of scalars from 
$k$ to $k^{\prime}$ also carries such an automorphism $\gamma$, in 
particular, $D_X^{\prime}$ does
(it sends a differential operator $P(\lambda)$ to $P(\lambda+1)$).

Similarly, $\sF^{\prime}$ has such an automorphism:
note that this is not an automorphism as a $D_X^{\prime}$-module,
but rather intertwines the standard action with the one obtained
by twisting by the automorphism $\gamma$ of $D_X^{\prime}$.
That is, $\gamma(P \cdot s) = \gamma(P) \cdot \gamma(s)$
for $P \in D_X^{\prime}$ and $s \in \sF^{\prime}$.

Define $\gamma$ on the $D_X^{\prime}$-module
$``f^{\lambda}"$ by setting:

\[
\gamma(``f^{\lambda}" \cdot g) = 
``f^{\lambda+1}" \cdot \gamma(g) \coloneqq 
``f^{\lambda}" \cdot f \cdot \gamma(g)
\]

\noindent for $g$ a function on $X^{\prime}$.
Again, this morphism intertwines the actions of $D_X^{\prime}$
up to the automorphism $\gamma$ of $D_X^{\prime}$.

Tensoring, we obtain an automorphism $\gamma$ of
$``f^{\lambda}" \cdot \sF^{\prime}$ with similar semi-linearity. 

By the semi-linearity, we have:

\[
\SS_{X^{\prime}}(``f^{\lambda}" \cdot s) = \gamma \cdot \SS_{X^{\prime}}(\gamma(``f^{\lambda}" \cdot s))
\]

\noindent where we are using $\gamma$ to indicate the
induced automorphism of $T^* X^{\prime}$.
In particular, we find that 
$\dim \SS_{X^{\prime}} (``f^{\lambda}" \cdot s) = \dim \SS_{X^{\prime}} (\gamma(``f^{\lambda}" \cdot s))$.

Now let $\sG = F_{\delta}^{GKS} (``f^{\lambda}" \cdot \sF^{\prime})$,\footnote{So
this is a different $\sG$ from Step \ref{st:gks}, i.e., we are applying
the same construction to a different $D$-module.}
where the GKS filtration is taken with $``f^{\lambda}" \cdot \sF^{\prime}$
considered as a $D_X^{\prime}$-module.
By the above and Theorem \ref{t:gks}, $``f^{\lambda}" \cdot s \in \sG$ if and 
only if $\gamma(``f^{\lambda}" \cdot s) \in \sG$.

For any $s \in \sF$ (as opposed to $\sF^{\prime}$),
we clearly have $\gamma(``f^{\lambda}" \cdot s) = f^{\lambda+1} \cdot s$.
Since $\sG$ is a lattice (by Step \ref{st:gks}), 
$\gamma^N(s) = f^{\lambda+N} s  \in \sG$ for $N \gg 0$. 
But by the above, this means that $s \in \sG$. Since 
$``f^{\lambda}" \cdot \sF^{\prime}$ is $k^{\prime}$-spanned by
such vectors, this means that $\sG = ``f^{\lambda}" \cdot \sF^{\prime}$, 
as desired.

\step

We now show that the result is true for our
original $\sF$. Let $s \in \sF$;
we want to show $\dim \SS_X(s) \leq \dim X+\delta$.

We now write $``f^{\lambda}" \cdot \sF$ for the corresponding
$D_X[\lambda]$-module (as opposed to the fiber over the generic
point in $\bA_{\lambda}^1$, which is what we called by this name
previously). Note that $``f^{\lambda} \cdot \sF" = \sF \otimes_k k[\lambda]$ as a $\sO_X[\lambda]$-module.

Let $\sF_0$ be the $D_X[\lambda]$ submodule generated by 
$``f^{\lambda}" \cdot s$. Give $\sF_0$ the filtration
$F_i \sF_0 = D_X^{\leq i}[\lambda] \cdot ``f^{\lambda}" \cdot s$,
where $D_X^{\leq i}$ are differential operators of order $\leq i$.

Then $\gr_{\dot}(\sF_0)$ is the structure
sheaf of some closed subscheme
$Z \subset T^*X \times \bA_{\lambda}^1$. 
We have seen that the base-change of $Z$ to the generic
point of $\bA_{\lambda}^1$ has dimension $\leq \dim X + \delta$,
so the same is true for its fibers at closed points with 
only finitely many possible exceptions.

Choose a negative integer $-N$ not among this finite number of
exceptions. Then the coherent $D_X$-module $\sF_0/(\lambda+N)$
has singular support contained in 
$Z \times_{\bA_{\lambda}^1} \{-N\}$, so has dimension
$\leq \dim X + \delta$. We have the obvious morphism
of $D_U$-modules (in particular, of $D_X$-modules):

\[
\begin{gathered}
\big(``f^{\lambda}" \cdot \sF\big)/(\lambda+N) \to \sF \\
\sum_{i = 0}^r ``f^{\lambda}" \cdot \sigma_i \lambda^i \mapsto  
\sum_{i = 0}^r f^{-N} \cdot \sigma_i \cdot (-N)^i 
\end{gathered}
\]

\noindent induces a map $\sF_0/(\lambda+N)$ to 
$\sF$ sending the generator to $f^{-N} s$. By functoriality
of the GKS filtration (or standard singular support analysis),
this means that $f^{-N} s \in F_{\delta}^{GKS} \sF$,
and since $F_{\delta}^{GKS} \sF$ is a $D_X$-module, this
means that $s \in F_{\delta}^{GKS} \sF$ as well.

\end{proof}
  
\subsection{Preservation of holonomic defect}

We now proceed to prove Theorem \ref{t:hol}. The argument 
is straightforward at this point, and we proceed in cases.

\subsection{}

First, we treat pushforwards along an open embedding $j:U \to X$.

For $X$ smooth, a Cech argument reduces us to
the case of a basic open, which is treated in 
Lemma \ref{l:b-fn}. (Recall from \S \ref{ss:defect-cplx} that
$D$-modules with holonomic defect $\delta$ are closed under cones.)

For possibly non-smooth $X$, note that the problem is Zariski
local, so we may assume $X$ is affine.
Take a closed embedding $X \subset \bA^N$.
If $U = X \setminus Z$, then 
we have $U \into \bA^N\setminus Z \into \bA^N$ with
the first map being closed and the second being open. Therefore,
this pushforward preserves holonomic defect. Clearly
this implies the result for the pushforward along $U \into X$.

\subsection{}

Next, we treat restrictions to closed subschemes.

Let $i:Z \into X$ be closed and let $j:U = X\setminus Z \into X$.
Then we have an exact triangle:

\[
i_{*,dR} i^!(\sF) \to \sF \to j_{*,dR} j^!(\sF) \xar{+1}.
\]

\noindent If $\sF$ has holonomic defect $\delta$, we have
shown the same for $j_{*,dR}j^!(\sF)$, so
$i_{*,dR}i^!(\sF)$ has holonomic defect $\delta$, which is
equivalent to $i^!(\sF)$ having holonomic defect $\delta$.

\subsection{}\label{ss:restr}

We can now show the result for restrictions in general.

If $f:X \to Y$ is smooth of relative
dimension $d$, then $f^{*,dR}[d] = f^![-d]$ commutes with
Verdier duality and is $t$-exact. Therefore, it commutes
with formation of the GKS filtration, and therefore 
preserves holonomic defect.

The case of general $f:X \to Y$ is immediately reduced to the
case of affine varieties (since holonomic defect is Zariski local).
We can find a commutative diagram:

\begin{equation}\label{eq:aff-f}
\vcenter{
\xymatrix{
X \ar[r] \ar[d]^f & \bA^{N_1} \ar[d]^g \\
Y \ar[r] & \bA^{N_2}
}}
\end{equation}

\noindent with the horizontal arrows being closed embeddings.
This reduces to the case where $X$ and $Y$ are smooth.

Then we can factor $f$ through the
graph as $X \to X \times Y \xar{p_1} Y$. The former map is a
closed embedding, and the latter is smooth because $X$ is.
We have treated each of these cases, so we obtain the result.

\subsection{}

Next, we treat pushforwards along a proper morphism
$f:X \to Y$ between smooth varieties.

This case does not need the work we have done so far.
Let $\sF \in D(X)^{\heart}$ with holonomic defect $\delta$
be given. By Lemma \ref{l:defect-colim}, we may
assume $\sF$ is coherent, so the hypothesis is that
$\sF$ has singular support $\SS_X(\sF)$ 
with dimension $\dim X + \delta$.

Recall that $\SS_Y(H^i(f_{*,dR}(\sF)))$ is bounded in
terms of $\SS_X(\sF)$. More precisely, if we take the diagram:

\[
\xymatrix{
T^* Y \underset{Y}{\times} X \ar[r]^{\alpha} \ar[d]^{\beta} & T^* X \\
T^* Y
}
\]

\noindent then the singular support of these cohomologies are
contained in $\alpha(\beta^{-1}\SS_X(\sF))$
(see e.g. \cite{kashiwara} Theorem 4.2).

Because $\SS(\sF)$ is coisotropic by \cite{gabber}, we have:

\[
\dim \alpha(\beta^{-1}\SS_X(\sF)) \leq 
\dim(\SS_X(\sF)) + \dim Y - \dim X
\]

\noindent by usual symplectic geometry. This immediately gives the claim.

\subsection{}

Now observe that preservation of holonomic defect under
pushforward along a general morphism
$f:X \to Y$ of smooth varieties follows: by Nagata and resolution 
of singularities,\footnote{Of course, one may 
easily use the more elementary
de Jong alterations instead.} 
we may find smooth $\ol{X}$ and a factorization:

\[
X \xar{j} \ol{X} \xar{\ol{f}} Y 
\]

\noindent of $f$ with $\ol{f}$ proper and $j$ an open embedding,
so we are reduced to our previous work.

\subsection{}

We can now treat a general pushforward 
along $f:X \to Y$ a morphism between possibly
singular varieties. 

Because we know pushforward along
open embeddings preserves holonomic defect, 
Cech reduces us to the case where $X$ and $Y$
are affine. Then we can find a commutative
diagram \eqref{eq:aff-f} as before. 
This reduces to the case with $X$ and $Y$ smooth, which
we have already treated.

\section{Cohomological bounds}\label{s:coh}

\subsection{}

The main result of this section says that $f_!$ is left $t$-exact 
for an affine morphism $f$. We also show that
for $i:X \to Y$ a closed embedding, $i^{*,dR}$ has cohomological
amplitude $\geq -\dim(Y)+\dim(X)$, i.e., 
$i^{*,dR}[-\dim(Y)+\dim(X)]$ is left $t$-exact.
Since $f_!$ and $i^{*,dR}$ are not defined on
every $D$-module (e.g., on non-holonomic ones), we
use the language of pro-categories to formulate this result.

\subsection{Pro-categories}\label{ss:pro-tstr}

For $\sC$ a\footnote{Really $\sC$ should be \emph{accessible}. 
Recall that this is a robust set-theoretic condition
satisfied by any small category and by any compactly generated
category. One should be aware that $\Pro(\sC)$ is 
almost never accessible itself.}
category, we have $\Pro(\sC)$ the corresponding
pro-category. If $\sC$ is a DG category,
$\Pro(\sC)$ is as well. If $\sC$ admits
small colimits, then so does $\Pro(\sC)$. 
For $F:\sC \to \sD$, there is an induced functor
$\Pro(\sC) \to \Pro(\sD)$, which we denote again by
$F$ where there is no risk for confusion.

For any functor $G:\sD \to \sC$ commuting with
finite colimits (e.g., a DG functor), the induced functor
$\Pro(\sD) \to \Pro(\sC)$ admits a left adjoint $F$.
We say that $F$ \emph{is defined} on an object
$\sF \in \sC$ if $F(\sF) \in \sD \subset \Pro(\sD)$. (This
coincides with the usual notion of a left adjoint being defined
on some object.)

If $\sC$ is a DG category equipped with a $t$-structure, the $\Pro(\sC)$
inherits one as well. It is characterized by the
equality $\Pro(\sC)^{\leq 0} = \Pro(\sC^{\leq 0})$.
Truncation functors are the pro-extensions of the truncation
functors on $\sC$. In particular, we find that $\sC$ is
closed under truncations and inherits its given $t$-structure.
We also find that 
$\Pro(\sC)^{\geq 0} = \Pro(\sC^{\geq 0})$:
if $\sF = \lim_i \sF_i \in \Pro(\sC)^{\geq 0}$, then
$\sF = \tau^{\geq 0} \sF = \lim_i \tau^{\geq 0} \sF_i$.

\subsection{Affine morphisms}

For $f:X \to Y$, we have the functor $f_!:\Pro(D(X)) \to \Pro(D(Y))$
left adjoint to $f^!$. 

\begin{thm}\label{t:aff}

For $f$ affine, the induced functor 
$f_!:D(X)^c \to \Pro(D(Y))$ is left $t$-exact.

\end{thm}

\begin{proof}

The problem is\footnote{Indeed, if $Y = U_1 \cup U_2$
with embeddings $j_i:U_i \into Y$ and $j_{12}:U_1 \cap U_2 \into Y$,
then for $\sG \in \Pro(D(Y))$ with $j_i^!(\sG) \in \Pro(D(U_i))^{\geq 0}$,
we want to see that $\sG \in \Pro(D(Y))^{\geq 0}$.

Note that:

\[
\sG = \Ker\big(j_{1,*,dR}j_1^!(\sG)\oplus j_{2,*,dR}j_2^!(\sG) \to
j_{12,*,dR}j_{12}^!(\sG) \big).
\]

\noindent Indeed, this follows by pro-extension from the
corresponding fact for usual $D$-modules. 
Since $t$-exact functors induce $t$-exact functors on pro-categories
as well, we obviously obtain the claim.
}
Zariski local on $Y$, so we may assume
$X$ and $Y$ are affine. 

Note that $D$-module pushforward along
closed embeddings remains fully-faithful on pro-categories:
the identity $i^!i_{*,dR} = \id$ induces the same for the pro-functors.
Therefore, the same argument as in \S \ref{ss:restr} allows
us to assume $X$ and $Y$ are smooth.

Recall that we have a Verdier duality equivalence
$\bD:D(X) \isom \Pro(D(X)^c)$ induced by the usual Verdier
duality equivalence $\bD:D(X)^c \isom D(X)^{c,op}$,
and similarly for $Y$.

We then claim that: 

\[
f_!(\sF) = \bD f_{*,dR} \bD (\sF).
\]

\noindent This follows formally from the fact that 
$f_{*,dR}$ and $f^!$ are \emph{dual} 
functors in the sense of \cite{dgcat},
but here is a direct proof anyway.
Note that in this
formula, $f_{*,dR} \bD(\sF) \in D(Y)$, and we are then using
$\bD$ to convert it to a pro-coherent object.
Since this object is pro-coherent, 
it suffices to observe that for 
$\sG \in D(Y)^c$, we have:

\[
\begin{gathered}
\Hom_{\Pro(D(Y)^c)}(\bD f_{*,dR} \bD (\sF), \sG) =
\Hom_{D(Y)}(\bD \sG, f_{*,dR} \bD (\sF)) =
\Gamma_{dR}(Y, f_{*,dR} \bD(\sF) \overset{!}{\otimes} \sG) = \\
\Gamma_{dR}(X, \bD(\sF) \overset{!}{\otimes} f^!(\sG)) = 
\Hom_{D(X)}(\sF,f^!(\sG)).
\end{gathered}
\]

\noindent Here $\Gamma_{dR}$ is the complex of de Rham
cochains of a $D$-module, and we are repeatedly using
the formula that if $\sF_1$ is coherent, then:

\[
\Hom_{D(X)}(\sF_1,\sF_2) = 
\Gamma_{dR}(X, \bD(\sF_1) \overset{!}{\otimes} \sF_2).
\]

Note that $\bD \sF$ carries the canonical
filtration with subquotients $H^{-j} \bD \sF[-j]$.

By Proposition \ref{p:dual-ss},
$H^{-j} \bD \sF$ has holonomic defect $j$. 
By Theorem \ref{t:hol},
$f_{*,dR} H^{-j} \bD \sF$ has holonomic defect $j$ as well.
Moreover, by affineness of $f$, this latter complex is in cohomological
degrees $\leq 0$. 

Note that by Proposition \ref{p:dual-ss}, if
$\sG \in D(Y)^{\heart}$ has holonomic defect $\delta$, then
$\bD \sG \in \Pro(D(Y)^c)$ is in cohomological degrees
$[-\delta,0]$: indeed, this immediately reduces to the coherent
case.

Therefore, $\bD H^{-k} f_{*,dR} H^{-j} \bD \sF$ is in cohomological
degrees $[-j,0]$ for every $k$, which means
$\bD \big(H^{-k} (f_{*,dR} H^{-j} \bD \sF)[k]\big)$ is in cohomological
degrees $[-j+k,k]$. This complex vanishes unless $k\geq 0$,
so $\bD f_{*,dR} H^{-j} \bD \sF$
is in cohomological degrees $\geq -j$.
Finally, this means that $\bD \big((f_{*,dR} H^{-j} \bD \sF)[j]\big)$
is in cohomological degrees $\geq 0$, so the same follows
for $\bD f_{*,dR} \bD(\sF) = f_!(\sF)$.

\end{proof}

\begin{rem}

More generally, this argument shows that
if $f_*:\QCoh(X) \to \QCoh(Y)$ has amplitude $\leq n$,
then $f_!:D(X)^c \to \Pro(D(Y))$ has amplitude $\geq -n$.

\end{rem}

\subsection{Closed embeddings}

Similarly, we have:

\begin{thm}

For $i:X \to Y$ a closed embedding, 
$i^{*,dR}:D(Y)^c \to \Pro(D(X))$ has
cohomological amplitude $\geq -\dim(Y)+\dim(X)$.

\end{thm}

\begin{proof}

The argument is the same as the above:
one writes $i^{*,dR} = \bD i^! \bD$ and applies 
Theorem \ref{t:hol} and Proposition \ref{p:dual-ss},
plus the fact that $i^!$ has amplitude $\leq \dim(Y) - \dim(X)$.

\end{proof}

\bibliography{bibtex}{}
\bibliographystyle{alphanum}

\end{document}